\documentclass[12pt]{article}
\usepackage{amsmath,amssymb,amsbsy,amsfonts,amsthm,latexsym,
amsopn,amstext,amsxtra,euscript,amscd}
\newtheorem{theorem}{Theorem}

\newtheorem{conj}[theorem]{Conjecture}

\newtheorem{cor}[theorem]{Corollary}
\newtheorem{prop}[theorem]{Proposition}

\newtheorem{thm}{Theorem}

\newtheorem{lem}[thm]{Lemma}

\def\\{\cr}
\def\({\left(}
\def\){\right)}
\def\[{\left[}
\def\]{\right]}
\def\<{\langle}
\def\>{\rangle}

\def\notdivides{\mathrel{\kern-3pt\not\!\kern3.5pt\bigm \mid }}
\def\nDiv{\nmid}

\begin{document}
\title{On the last digit and the last non-zero digit of $n^n$ in base $b$}

\author{
{\sc Jos\'{e} Mar\'{i}a Grau}\\
{Departamento de Matem\'{a}ticas}\\
{Universidad de Oviedo}\\
{Avda. Calvo Sotelo, s/n, 33007 Oviedo, Spain}\\
{grau@uniovi.es}\\
{\sc Antonio M. Oller-Marc\'{e}n}\\
{Centro Universitario de la Defensa}\\
{Ctra. Huesca s/n, 50090 Zaragoza, Spain}\\
{oller@unizar.es}}

\date{}

\maketitle

\newpage
\begin{abstract}
In this paper we study the sequences defined by the last and the last non-zero digits of $n^n$ in base $b$. For the sequence given by the last digits of $n^n$ in base $b$, we prove its periodicity using different techniques than those used by W. Sierpinski and R. Hampel. In the case of the sequence given by the last non-zero digits of $n^n$ in base $b$ (which had been studied only for $b=10$) we show the non-periodicity of the sequence when $b$ is an odd prime power and when it is even and square-free. We also show that if $b=2^{2^s}$ the sequence is periodic and conjecture that this is the only such case.
\end{abstract}

\section{Introduction}

The study of the last digit of the elements in a sequence is a recurrent topic in Number Theory. In this sense, one of the most studied sequences is, of course, the Fibonacci sequence which was already studied by Lagrange 
 observing that the last digit of the Fibonacci sequence repeats with period 60 (see \cite{Lagrange}). In any base $b$, the sequence of Fibonacci modulo $b$ is also periodic \cite{Wall} and the periods $\pi(b)$ for each base $b$ (see \cite{Rob,Morris} for some of their properties) are called Pisano periods (Sloane's OEIS A001175). These periods have been conjectured to satisfy the relation $\pi(p^e) = p^{e-1}\pi(p)$ which is called Wall's conjecture and that has been verified for primes up to $10^{14}$. Primes for which this relation fails (if any exists) are called Wall-Sun-Sun primes. 

There are many other examples of works of similar orientation. In \cite{Walter}, for instance, the last decimal digit of $\binom{2n}{n}$ and $\sum \binom{n}{i}\binom{2n-2i}{n-i}$ is explicitly computed and D.B. Shapiro and S.D. Shapiro show in \cite{Daniel}, among other results, that the sequence $k, k^k, k^{k^{k}}, \dots, k\uparrow\uparrow n,\dots$ (mod $b$) is eventually constant.

In this paper we focus on the sequence $n^n$. The study of the residues of this sequence was started by W. Sierpinski who, in his 1950 paper \cite{sier}, proved that the last digits of the numbers $n^n$ form a periodic sequence whose shortest period consists of 20 terms. More generally, it was proved that, for every positive integer $b$, the sequence $S_b (n)$ consisting of the residues mod $b$ of the numbers $n^n$ form an infinite, eventually periodical, sequence. In 1955, R. Hampel (see \cite{Hampel}) proved that the period of $S_b(n)$ (Sloane's OEIS A174824) is $\textrm{lcm}(b,\lambda(b))$, where $\lambda$ is the Carmichael function. Moreover, he proved that if $b=\prod_{i=1}^t p_i^{s_i}$, the sequence is periodic if and only if $s_i\leq p_i $ and that periodicity starts with the maximum of the numbers $\eta_{i}:=1-p_i (1+\lceil -\frac{s_i}{p_i}\rceil)$ for $i=1,\cdot\cdot\cdot,t$. These results were established first in the prime case (by Sierpinski), then in the prime power case, and finally in general. The methods of the proof lie in the theory of linear congruences and frequent use is made of the Euler-Fermat congruence and of the properties of primitive roots. It seems remarkable to us the fact that this work by Hampel was not cited in recent work on this topic, such \cite{Cro1,Cro2,Euler,Dresden,Dresden2}.

In a somewhat different direction we find the works by R. Crocker \cite{Cro1,Cro2} and L. Somer \cite{somer} where they study the number of residues (mod $p$) of $n^n$, for $n$ between $1$ and $p$. More recently the interest on the sequence $n^n$ was revived by G. Dresden in \cite{Dresden}, where he established the non-periodicity of the last non-zero digit of the decimal expansion of this sequence and in \cite{Dresden2}, where he proves that the number formed by this digits is transcendental.

Our paper is organized as follows. In the second section we revisit, using different techniques, the work by Sierpinski and Hampel. In the third section we focus on the last non-zero digit of $n^n$ in base $b$. In particular we establish the non-periodicity of this sequence when $b$ is an odd prime power or an even square-free integer. We also show that if $b=2^{2^s}$ the sequence is periodic and conjecture that this is the only such case.

\section{The last digit of $n^n$ in base $b$}

The results that we present in this section were already proved in \cite{Hampel,sier}. We revisit then using quite different techniques.

We will start with some notation. Given $n,b\in\mathbb{N}$ we consider the following functions:
$$\mathcal{H}(n):=\textrm{lcm}(n,\lambda(n)),$$
$$S_b(n):=n^n\ \textrm{(mod b)}.$$

Observe that $S_b(n)$ gives the last digit of $n^n$ in base $b$. We are interested in studying the behavior of this sequence. A first step in this direction is given in the following proposition. This proposition not only determines the eventual periodicity of $S_b(n)$, but also the values that break the periodicity. This question was not studied by Hampel in \cite{Hampel}.

\begin{prop}
For every $b\in\mathbb{N}$ let $S_b(n)$ be the sequence defined above. Let $M\in\mathbb{N}$ and put $M=\displaystyle{\prod_{i=1}^t p_i^{k_i}}$ with $t>0$ its prime power decomposition. Then $S_b(M)\neq S_b(M+\mathcal{H}(b))$ if and only if $p_i^{k_iM+1}$ divides $b$ for some $i\in\{1,\dots,t\}$.
\end{prop}
\begin{proof}
Put $b=p_1^{a_1}\cdots p_t^{a_t}q_1^{r_1}\cdots q_s^{r_s}$ the prime-power decomposition of $b$ ($q_i\neq p_j$). We have that $M+\mathcal{H}(b)\equiv M$ (mod $b$). Also, since
$\lambda(q_i^{r_i}) \mid \mathcal{H}(b)$, we have that $M^{\mathcal{H}(b)}\equiv 1$ (mod $b$) and it follows that $M^M\equiv (M+\mathcal{H}(b))^{M+\mathcal{H}(b)}$ (mod $q_i^{r_i}$) for every $i\in\{1,\dots,s\}$. As a consequence
$M^M\not\equiv (M+\mathcal{H}(b))^{M+\mathcal{H}(b)}$ (mod $b$) if and only if $M^M\not\equiv M^{M+\mathcal{H}(b)}$ (mod $p_i^{a_i}$) for some $i\in\{1,\dots, t\}$. Clearly, this happens if and only if
$M^M(M^{\mathcal{H}(b)}-1)\not\equiv 0$ (mod $p_i^{a_i}$). But, since $p_i$ does not divide $M^{\mathcal{H}(b)}-1$, this happens if and only if $M^M\neq 0$ (mod $p_i^{a_i}$).
Finally, $M^M=\prod_{i=1}^t p_i^{k_iM}\not\equiv 0$ (mod $p_i^{a_i}$) if and only if $a_i\geq k_iM+1$; i.e., if and only if $p_i^{k_iM+1} \mid b$.
\end{proof}

This result clearly implies that the sequence $S_b(n)$ is eventually periodic with its period being a divisor of $\mathcal{H}(b)$. The next results are devoted to show that the period is exactly $\mathcal{H}(b)$.

\begin{prop}
If $S_b(n)=S_b(n+T)$ for every $n\geq n_0$, then $b$ divides $T$.
\end{prop}
\begin{proof}
We can choose $n\equiv 0$ (mod $b$) and it follows that $T^{n+T}\equiv 0$ (mod $b$). This implies that $\textrm{rad}(b) \mid T$.

Now, $n^n\equiv (n+T)^{n+T}$ (mod rad($b$)) for every $n\geq n_0$, so we have that $n^n\equiv n^{n+T}$ (mod rad($b$)). We can choose $n$ such that $\gcd(n,b)=1$ so that $n^T\equiv 1$ (mod rad($b$)). From this, it follows that $\varphi(\textrm{rad}(b)) \mid T$; i.e., if $b=p_1^{b_1}\cdots p_s^{b_s}$ then $(p_1-1)\cdots (p_s-1) \mid T$.

If we choose $n\equiv 1$ (mod $b$), then it follows that $(T+1)^{n+T}\equiv 1$ (mod $b$). Since rad($b$) divides $T$, we have that $\gcd(T+1,b)=1$ and, consequently, that $\gcd(T+1,p_i^{b_i})=1$. Thus $(T+1)^{\gcd(\varphi(p_i^{b_i}),n+T)}\equiv 1$ (mod $p_i^{b_i}$).

Assume that $p_i \mid n+T$. Then, since $p_i \mid T$ it follows that $p_i \mid n$. This is a contradiction because $p_i \mid n-1$ and we get that $\gcd(\varphi(p_i^{b_i}),n+T)=\gcd(n+T,p_i-1)$. On the other hand, it can be easily seen that, in our conditions, $\gcd(n+T,p_i-1)=\gcd(n,p_i-1)$. We have thus seen that $(T+1)^{\gcd(n,p_i-1)}\equiv 1$ (mod $b$) for every $n\equiv 1$ (mod $b$).

We can now choose $n=K(p_1-1)\cdots (p_s-1)b+1$ with $k$ such that $n\geq n_0$. It is clear that $n\equiv 1$ (mod $b$) and, moreover, $\gcd(n,p_i-1)=1$. Thus we obtain that $(T+1)\equiv 1$ (mod $b$) and the result follows.
\end{proof}

\begin{cor}
If $S_b(n)=S_b(n+T)$ for every $n\geq n_0$, then $\lambda(b)$ divides $T$.
\end{cor}
\begin{proof}
In the previous proposition be have seen that $b \mid T$. Thus, $S_b(n)=S_b(n+T)$ implies that $n^n\equiv (n+T)^{n+T}\equiv n^{n+T}$ (mod $b$) and, consequently, that $n^n(n^T-1)\equiv 0$ (mod $b$) for every $n\geq n_0$. There is no problem in choosing $n$ such that $\gcd(n,b)=1$ and then $n^T\equiv 1$ (mod $b$) for every $n\geq n_0$ coprime to $b$. This clearly completes the proof.
\end{proof}

\begin{cor}
Given $b\in\mathbb{N}$, the sequence $S_b(n)$ is eventually periodic of period $\mathcal{H}(b)$.
\end{cor}
\begin{proof}
Due to Proposition 1, the sequence $S_b(n)$ is eventually periodic and its period must divide $\mathcal{H}(b)$. Now, let $T$ be the period. Proposition 2 and Corollary 1 imply that $b$ and $\lambda(b)$ both divide $T$ and hence the result.
\end{proof}

We have seen that $S_b(n)$ is eventually periodic. It is also interesting to study in which cases this sequence is periodic.

\begin{prop}
Let $b=\displaystyle{\prod_{i=1}^t p_i ^{s_i}}$. The sequence $S_b(n)$ is periodic if and only if $s_i \leq p_i$ for every $i\in\{1, \dots,t\}$.
\end{prop}
\begin{proof}
Assume that $S_b(n)$ is periodic with period $\mathcal{H}(b)$. Then $S_b(p_i)=S_b(p_i+\mathcal{H}(b))$ for every $i$ so the Proposition 1 implies that $p_i^{p_i+1}$ does not divide $p_i^{s_i}$; i.e., $p_i\geq s_i$ for every $i$ as claimed.

Conversely, assume that $s_i\leq p_i$ for every $i$.

Let $n=p_1^{k_1}\cdots p_t^{k_t}$ be an integer such that the primes in its decomposition are the same than those in the decomposition of $b$. If $i\in\{1,\dots,t\}$ is such that $k_i\neq 0$ we have that $s_i\leq p_i\leq k_ip_i^{k_i}\leq k_iM$ so by Proposition 1 again $S_b(n)=S_b(n+\mathcal{H}(b))$.

On the other hand, if $\gcd(n,b)=1$ we have that $S_b(n)=S_b(n+\mathcal{H}(b))$ since $\lambda(b)$ divides $\mathcal{H}(b)$.

To finish the proof it is enough to observe that every $n\in\mathbb{N}$ can be written in the form $n=n_1n_2$ with $\gcd(n_2,b)=1$ and to reason like in the previous cases.
\end{proof}

\section{The last non-zero digit of $n^n$ in base $b$}

In the previous section we have proved that the sequence $S_b(n)=n^n$ (mod $b$) given by the last digit of $n^n$ is eventually periodic. For instance, if $b=3$ the first elements of $S_3(n)$ are:
$$1, 1, 0, 1, 2, 0, 1, 1, 0, 1, 2, 0, 1, 1, 0, 1, 2, 0, 1,1, 0, 1, 2, 0, 1, 1, 0,1,2,0,1,\dots$$
and the period is $(1,1,0,1,2,0)$.
We can see that there are many zeros in the previous sequence, in fact if $3\mid n$ then clearly $S_3(n)=0$. We wonder what will happen if we consider the sequence given by the last non-zero digit of $n^n$ instead. In this case the 0's will disappear and they will be replaced by 1 or 2 and periodicity could  be possibly broken. For the case $b=10$ it is well-known (see \cite{Dresden}) that the sequence given by the last non-zero digit of $n^n$ in base 10 is not eventually periodic. In this section we will focus on the behavior of this sequence for some choices of $b$. In particular we will study the case when $b$ is a square-free even integer and when it is a prime power.

Before we proceed, we will introduce some notation. In what follows $L_b(n)$ will denote the last non-zero digit of $n$ in base $b$. For every $b\in\mathbb{N}$ we will consider the sequence $\mathfrak{S}_b(n):=L_b(n^n)$; i.e., $\mathfrak{S}_b(n)$ is the last non-zero digit of $n^n$. Observe that if $b\nDiv n$, then $L_b(n)\equiv n$ (mod $b$).

\subsection{The even square-free case}

We will show in this subsection that $\mathfrak{S}_b(n)$ is not eventually periodic when $b$ is a square-free even integer. Our proof will be simpler than the one given in \cite{Dresden} for the case $b=10$.

We will start with a series of technical lemmas.

\begin{lem}\label{L1}
Let $a(n)$ be a sequence such that $a(n)\in\{e_1\dots,e_r\}$ for every $n\in\mathbb{N}$. If $a(n)$ is eventually periodic, then the set $\Theta(e_i):=\{n:a(n)=e_i\}$ is (possibly with the exception of a finite number of elements) the union of a finite number of arithmetic sequences.
\end{lem}
\begin{proof}
Assume that $a(n)$ is periodic with period $T$ and put $n_{0,i}=\min\Theta(e_i)$. Clearly $n_{0,i}+kT\in\Theta(e_i)$ for every $k$. Let $\{n_{1,i},\dots,n_{m_i,i}\}=\Theta(e_i)\cap(n_{0,i},n_{0,i}+T)$. We claim that
$$\Theta(e_i)=\bigcup_{j=0}^{m_i}\{n_{j,i}+kT:k\in\mathbb{N}\}.$$

For let $n\in\Theta(e_i)$. Then there must exist $k\in\mathbb{N}$ such that $n_{0,i}+kT\leq n<n_{0,i}+(k-1)T$. But in this case $n_{0,i}\leq n-kT<n_{0,i}+T$ so $n-kT=n_{j,i}$ for some $j\in\{0,\dots,m_i\}$ as claimed.

If $a(n)$ is not periodic, but eventually periodic, we can reason in the same way but a finite number of initial terms must be considered separately and the result follows.
\end{proof}

\begin{lem}
Let $b$ be an even square-free integer and put $b=2m$. Then $\Theta(m):=\{n:\mathfrak{S}_b(n)=m\}=\{n:L_{b}(n)=m\}$.
\end{lem}
\begin{proof}
Let $n=b^rn'$ with $r\geq 0$ and $b$ not dividing $n'$.

$\mathfrak{S}_b(n)=m$ if and only if $(n')^n\equiv m$ (mod $b$). This implies that $(n')^m\equiv 0$ (mod $m$) and $(n')^n\equiv 1$ (mod $2$) simultaneously. But, $b$ being square-free, it follows that $n'\equiv 0$ (mod $m$) and $n'\equiv 1$ (mod $2$); i.e., $m\equiv (n')^n\equiv n'$ (mod $b$). Thus $L_b(n)=L_b(n')\equiv n'\equiv m$ (mod $b$).

Since the steps above are reversible the proof is complete.
\end{proof}

Let us now define the following family of sets:
$$\mathcal{C}_i:=\{mb^{i-1}+kb^i:k\in\mathbb{N}\}.$$
Observe that $\mathcal{C}_i\subset \Theta(m)$ and the previous lemma implies that
$$\Theta(m)=\bigcup_{i\geq 1}\mathcal{C}_i.$$
We are now in the conditions to prove the following result.

\begin{prop}
For every even and square-free integer $b$, the sequence $\mathfrak{S}_b(n)$ is not eventually periodic.
\end{prop}
\begin{proof}
Assume that $\mathfrak{S}_b(n)$ is eventually periodic. Then, due to Lemma \ref{L1} it follows that (with the exception of a finite number of elements) the set $\Theta(m)$ is a finite union of arithmetic sequences; i.e., $\displaystyle{\Theta(m)=\bigcup_{i=1}^r A_i}$. To prove the result we can put aside, without loss of generality, the finite number of elements which do not lie in this finite union of arithmetic sequences.

Let $a_{0,i}=\min A_i$ so that $A_i=\{a_{0,i}+kd_i:k\in\mathbb{N}\}$ for every $i$. If we denote by $a_{k,i}=a_{0,i}+kd_i$, the following facts should be clear from the very definition of the sets $\mathcal{C}_i$:
\begin{enumerate}
\item If $a_{k,i},a_{k+1,i}\in\mathcal{C}_j$ for some $k\in\mathbb{N}$, then $b^j \mid d_i$ and $a_{h,i}\in\mathcal{C}_j$ for every $h\geq k$.
\item If $a_{k,i}\in\mathcal{C}_{j_1}$ and $a_{k+1,i}\in\mathcal{C}_{j_2}$, then $j_2>j_1$ because, otherwise, $a_{k+1,i}\not\in\Theta(m)$.
\item If $a_{k,i}\in\mathcal{C}_{j_1}$ and $a_{k+1,i}\in\mathcal{C}_{j_2}$ with $j_2>j_1$, then $a_{k+2,i}\in C_{j_1}$. Consequently, we can apply the previous point to find a contradiction.
\end{enumerate}
The three points above show that if $\displaystyle{\Theta(m)=\bigcup_{i=1}^r A_i}$, then each $A_i$ is eventually contained in some fixed $\mathcal{C}_{j(i)}$. This clearly contradicts the fact that $\displaystyle{\Theta(m)=\bigcup_{i\geq 1}\mathcal{C}_{i}}$ and the proof is finished.
\end{proof}

\subsection{The prime power case}
In this section we focus on the behavior of the sequence $\mathfrak{S}_{p^t}(n)$ with $p$ a prime and $t\geq 1$. Since this situation is rather different from the situation of the previous section we will have to use different techniques here. In fact we have to study the case $t=1$ separately.

\subsubsection{The case $t=1$}
To study the behavior of the sequence $\mathfrak{S}_{p}(n)$ for every prime $p$ we will make use of some kind of ``fractality'' of this sequence, which is  established in the following lemma.

\begin{lem}\label{fractal}
If $p$ is a prime, then $\mathfrak{S}_p(n)=\mathfrak{S}_p(pn)$.
\end{lem}
\begin{proof}
If $p\nDiv n$, since $p\nDiv n^n$, we have that $\mathfrak{S}_p(n)=L_p(n^n)\equiv n^n$ (mod $p$). On the other hand, $\mathfrak{S}_p(pn)=L_p(p^{pn}n^{pn})=L_p(n^{pn})\equiv n^{pn}\equiv n^n$ (mod $p$).

Now, if $n=p^mn'$ with $p\nDiv n'$, then:
\begin{align*}
\mathfrak{S}_{p}(pn)&=L_p(p^{pn}n^{pn})=L_p(n^{pn})=L_p(p^{mpn}(n')^{pn})=\\&=L_p((n')^{pn})\equiv (n')^{pn}=(n')^{p^{m+1}n'}\equiv (n')^{n'}\ \textrm{(mod $p$)},
\end{align*}
while:
$$\mathfrak{S}_p(n)=L_p(n^n)=L_p(p^{mn}(n')^{n})=L_p((n')^{p^mn'})\equiv (n')^{p^mn'}\equiv (n')^{n'}\ \textrm{(mod $p$)}$$
and hence the result.
\end{proof}

The previous lemma gives us, in addition, some information about the period of $\mathfrak{S}_p(n)$, if it exists.

\begin{lem}
If the sequence $\mathfrak{S}_p(n)$ is eventually periodic of period $T$, then $p\nDiv T$
\end{lem}
\begin{proof}
If $p \mid T$ we have:
$$\mathfrak{S}_p(n)=\mathfrak{S}_p(pn)=\mathfrak{S}_p(pn+T)=\mathfrak{S}_p(pn+pT')=\mathfrak{S}_{p}(n+T')$$
with $T'<T$, a contradiction.
\end{proof}

The next proposition proves the non-periodicity of $\mathfrak{S}_p(n)$ if $p$ is odd.

\begin{prop}
If $p$ is an odd prime, the sequence $\mathfrak{S}_p(n)$ is not eventually periodic.
\end{prop}
\begin{proof}
Assume, on the contrary, that the sequence is eventually periodic; i.e., that $\mathfrak{S}_p(n)=\mathfrak{S}_p(n+T)$ eventually, with minimal $T$.

Take $n=p^mn'$ with $p\nDiv n'$. Like in Lemma 3 above, $\mathfrak{S}_{p}(n)\equiv (n')^{n'}$ (mod $p$). Now, since $p \mid n$ but $p\nDiv T$, it follows that $p\nDiv (n+T)^{n+T}$ and thus:
$$\mathfrak{S}_p(n+T)\equiv (n+T)^{n+T}\equiv T^{n+T}\equiv T^{n'+T}\ \textrm{(mod $p$)}.$$
We have thus seen that for every $n'$ such that $p\nDiv n'$, $T^{n'+T}\equiv (n')^{n'}$ (mod $p$).

If we take $n'=1$ it follows that $T^{T+1}\equiv 1$ (mod $p$). If we take $n'=p-1$ (recall that $p\neq 2$) it follows that $T^T\equiv 1$ (mod $p$). This facts together imply that $T\equiv 1$ (mod $p$) but this would imply that $(n')^{n'}\equiv 1$ (mod $p$) for every  $n'$ with $p\nDiv n'$. This is a contradiction and the proof is complete.
\end{proof}

To complete the study in this case it is enough to observe that $\mathfrak{S}_2(n)$ is obviously constant with $\mathfrak{S}_2(n)=1$ for every $n\in\mathbb{N}$ because, in base 2, the last non-zero digit of any number is 1.

\subsubsection{The case $t>1$}

Now, we turn to the sequence $\mathfrak{S}_{p^t}(n)$ with $p$ an odd prime and $t>1$. In this case we have the following analogue of Lemma \ref{fractal} to describe the ``fractality'' of $\mathfrak{S}_{p^t}(n)$.

\begin{lem}
Let $p$ be a prime and let $t>1$ be any integer. Then $\mathfrak{S}_{p^t}(p^tn)=\mathfrak{S}_{p^t}(p^{t+\varphi(t)}n)$ for every $n\in\mathbb{N}$.
\end{lem}
\begin{proof}
Put $n=p^mn'$ with $m\geq 0$ and $p\nDiv n'$. Then:
\begin{align*}\mathfrak{S}_{p^t}(p^tn)&=\mathfrak{S}_{p^t}(p^{m+t}n')=L_{p^t}\left(p^{(m+t)p^tn}(n')^{p^tn}\right)=\\&=L_{p^t}\left(p^{mp^tn}(n')^{p^tn}\right)\equiv p^{\alpha}(n')^{p^tn}\ \textrm{(mod $p^t$)},\end{align*}
where $\alpha\in\{0,\dots,t-1\}$ is the class of $mp^tn$ modulo $t$.

On the other hand it can be seen in the same way that:
$$\mathfrak{S}_{p^t}(p^{t+\varphi(t)}n)\equiv p^{\beta}(n')^{p^{t+\varphi(t)}n}\ \textrm{(mod $p^t$)},$$
where $\beta\in\{0,\dots,t-1\}$ is the class of $mp^{t+\varphi(t)}n$ modulo $t$.

Now, to finish the proof it is enough to see that $p^{\alpha}(n')^{p^tn}\equiv p^{\beta}(n')^{p^{t+\varphi(t)}n}$ (mod $p^t$). Obviously $p^{t+\varphi(t)}\equiv p^t$ (mod $t$), thus $\alpha=\beta$ and since $p^{t+\varphi(t)}=p\frac{p^{\varphi(t)}-1}{p-1}\varphi(p^t)+p^t$ it follows (recall that $p\nDiv n'$) that $(n')^{p^{t+\varphi(t)}}\equiv (n')^{p^t}$ (mod $p^t$) and we are done.
\end{proof}

Let us now define the sequence $S(n):=\mathfrak{S}_{p^t}(p^tn)$. The following result summarizes some properties of this sequence.

\begin{lem}\label{propi}
Let $p$ be a prime and let $S(n)$ the sequence defined above. Then, the following properties hold:
\begin{enumerate}
\item[i)] $S(n)=S(p^{\varphi(t)}n)$ for every $n\in\mathbb{N}$.
\item[ii)] If $S(n)$ is (eventually) periodic of period $T$, then $p\nDiv T$.
\item[iii)] If $\mathfrak{S}_{p^t}(n)$ is (eventually) periodic of period $T$, then $S(n)$ is also (eventually) periodic and its period divides $T$.
\end{enumerate}
\end{lem}
\begin{proof}
\begin{enumerate}
\item[i)] This is the previous lemma.
\item[ii)] If $p\mid T$ then $T=pT'$ and we have that $S(n)=S(p^{\varphi(t)}n)=S(p^{\varphi(t)}n+p^{\varphi{t}-1}T)=S(p^{\varphi(t)}n+p^{\varphi(t)}T')=S(n+T')$ with $T'<T$, a contradiction.
\item[iii)] Let $T$ be the period of $\mathfrak{S}_{p^t}(n)$. Then $S(n)=\mathfrak{S}_{p^t}(p^tn)=\mathfrak{S}_{p^t}(p^tn+p^tT)=S(n+T)$ as claimed.
\end{enumerate}
\end{proof}

As a consequence of the previous lemma, to prove that $\mathfrak{S}_{p^t}(n)$ is not eventually periodic it is enough to see that neither is $S(n)$.

\begin{prop}
Let $p$ be a prime and $S(n)=\mathfrak{S}_{p^t}(p^tn)$. If $S(n)$ is eventually periodic, then $t=p^s$.
\end{prop}
\begin{proof}
We know by hypothesis that $S(n)=S(n+T)$ for every $n\geq n_0$ and put $n=p^mn'$ with $p\nDiv n'$ as usual. Note that since we are dealing with eventual periodicity, we can assume without loss of generality that $m\geq t$. Then:
\begin{align*}
S(n)&=\mathfrak{S}_{p^t}(p^{m+t}n')=L_{p^t}\left(p^{(m+t)p^tn}(n')^{p^tn}\right)=L_{p^t}\left(p^{mnp^t}(n')^{p^tn}\right)\equiv\\ &\equiv p^{\alpha}(n')^{p^tn}\ \textrm{(mod $p^t$)},
\end{align*}
where $\alpha\in\{0,\dots,t-1\}$ is the class of $mnp^t$ modulo $t$.

On the other hand,
\begin{align*}S(n+T)&=\mathfrak{S}_{p^t}(p^tn+p^tT)=L_{p^t}\left((p^tn+p^tT)^{p^tn+p^tT}\right)=\\&=L_{p^t}\left((n+T)^{p^tn+p^tT}\right)\equiv (n+T)^{p^tn+p^tT}\equiv T^{p^t(n+T)}\ \textrm{(mod $p^t$)},\end{align*}
since $p^t\mid n$ because we have chosen $m\geq t$.

Thus, we have seen that for every $n_0\leq n=p^mn'$ with $p\nDiv n'$, if $\alpha$ is the class of $mnp^t$ modulo $t$, then:
$$T^{p^t(n+T)}\equiv p^{\alpha}(n')^{p^tn}\ \textrm{(mod $p^t$)}.$$

Clearly, if $t$ is not a power of $p$, we can choose $m$ and $n$ such that $\alpha\neq 0$ so it follows that $p$ divides $T^{p^t(n+T)}$, a contradiction.
\end{proof}

Due to the previous proposition we only have to worry about the case $\mathfrak{S}_{p^{p^s}}(n)$. We will see that if $p$ is odd, this sequence is not eventually periodic.

\begin{prop}
Let $p$ be an odd prime and $S(n)=\mathfrak{S}_{p^t}(p^tn)$. Then the sequence $S(n)$ is not eventually periodic.
\end{prop}
\begin{proof}
We can reason in a similar way to that in the previous proposition, but in this case $\alpha=0$ necessarily. Hence, using the same notation as in the previous result, we get $T^{p^{p^s}(n+T)}\equiv (n')^{p^{p^s}n}$ (mod $p^{p^s}$). We can choose $n'=1$ so it follows that $T^{p^{p^s}(p^m+T)}\equiv 1$ (mod $p^{p^s}$) and, consequently, that $T^{T+1}\equiv 1$ (mod $p$). If we now choose $n'=p-1$ (recall that $p\neq 2$) we get $T^{p^{p^s}(p^m(p-1)+T)}\equiv 1$ (mod $p^{p^s}$) and, consequently, that $T^{T}\equiv 1$ (mod $p$). Putting these two results together we get that $T\equiv 1$ (mod $p$) so $T^{p^s}\equiv 1$ (mod $p^{p^s}$).

Then, we have that $(n')^{p^{p^s}n}\equiv 1$ (mod $p^{p^s}$) for every $n'$ such that $p\nDiv n'$; i.e., $(n')^{n'}\equiv 1$ (mod $p$) for every $n'$ such that $p\nDiv n'$. Clearly this is impossible in $p\neq 2$ and the proof is complete.
\end{proof}

So, it only remains to study the sequence $\mathfrak{S}_{2^{2^s}}(n)$.

\begin{prop}
Let $b=2^{2^s}$. Then $\mathfrak{S}_b(n)=\mathfrak{S}_b(n+b)$ for every $n\in\mathbb{N}$.
\end{prop}
\begin{proof}
First of all we consider the case $b\nDiv n$. In this case, since $b\nDiv n+b$ it follows that $\mathfrak{S}_b(n)\equiv n^n$ (mod $b$) and $\mathfrak{S}(n+b)\equiv (n+b)^{n+b}\equiv n^{n+b}$ (mod $b$). We now consider two possibilities:
\begin{itemize}
\item[i)] If $n$ is odd $n^b\equiv 1$ (mod $b$) because $\varphi(b)=\varphi(2^{2^s}) \mid 2^{2^s}=b$. Thus $n^{n+b}\equiv n^n$ (mod $b$) and we are done.
\item[ii)] If $n$ is even we put $n=2^mn'$ with $n'$ odd and $m<2^s$. Thus, $n^{n+b}=2^{m2^{2^s}}(n')^{2^{2^s}}n^n$. Since $s<2^s$ it follows that $2^{m2^{2^s}}$ is a power of $b$ and consequently $L_b\left(2^{m2^{2^s}}\right)=1$. Moreover, since $(n')^{b}\equiv 1$ (mod $b$) it also follows that $L_b\left((n')^{b}\right)=1$ so:
\begin{align*}\mathfrak{S}_b(n+b)&=L_b(n^{n+b})\equiv L_b\left(2^{m2^{2^s}}\right)L_b\left((n')^{b}\right)L_b(n^n)=L_b(n^n)=\\&=\mathfrak{S}_b(n)\ \textrm{(mod $b$)}\end{align*}
an the proof is complete in this case.
\end{itemize}

Now, we have to consider the case when $b \mid n$; i.e., when $n=b^an'=2^{a2^s}n'$ with $b\nDiv n'$. In this case $\mathfrak{S}_b(n)=L_b(n^n)=L_b\left((n')^{2^{a2^s}n'}\right)$. Now, if $n'$ is odd we have that $L_b\left((n')^{2^{a2^s}n'}\right)(n')^{2^{a2^s}n'}\equiv 1$ (mod $b$). On the other hand, if $n'$ is even; i.e., $n'=2^mn''$ with $n''$ odd and $m<2^s$ we have that:
$L_b\left((n')^{2^{a2^s}n'}\right)=L_b\left(2^{amn'2^{2^s}}(n'')^{2^{an'2^s}}\right)=L_b\left((n'')^{2^{an'2^s}}\right)\equiv 1$ (mod $b$). Thus, we have seen that if $b \mid n$, then $\mathfrak{S}_b(n)=1$. We will have to compute now $\mathfrak{S}_b(n+b)$ in this case.

To do so, observe that
\begin{align*}\mathfrak{S}_b(n+b)&=L_b\left((n+b)^{n+b}\right)=L_b\left(\left(2^{2^s}(2^{2^s(a-1)}n'+1)\right)^{n+2^{2^s}}\right)=\\&=L_b\left(\left(2^{2^s(a-1)}n'+1\right)^{n+2^{2^s}}\right),\end{align*}
and two cases arise:
\begin{itemize}
\item[i)]
If $a>1$ then $b\nDiv 2^{2^s(a-1)}n'+1$ and thus $L_b\left(\left(2^{2^s(a-1)}n'+1\right)^{n+2^{2^s}}\right)\equiv \left(2^{2^s(a-1)}n'+1\right)^{n+2^{2^s}}\equiv 1$ (mod $b$).
\item[ii)]
If $a=1$ we must compute $L_b\left((n'+1)^{n+2^{2^s}}\right)$ and we have two sub cases:
\begin{itemize}
\item[ii1)]
If $n'+1$ is odd, then $(n'+1)^{n+2^{2^s}}\equiv 1$ (mod $b$).
\item[ii2)]
If $n'+1$ is even, $n'+1=2^mn''$ with $n''$ odd an clearly
$$L_b\left((n'+1)^{n+2^{2^s}}\right)=L_b\left((n'')^{2^{2^s}(n'+1)}\right)=1.$$
\end{itemize}
\end{itemize}

Thus, we have seen that if $b \mid n$, then $\mathfrak{S}_b(n+b)=1=\mathfrak{S}_b(n)$ and the proof is completely finished.
\end{proof}

After all the work done, we have proved the following result.

\begin{thm}
The sequence $\mathfrak{S}_{p^t}(n)$ is eventually periodic if and only if $p=2$ and $t=2^s$ for some $s\in\mathbb{N}$ and, in that case, it is periodic.
\end{thm}

\section{An ending conjecture}
The techniques that we have used in this paper have not been useful in order to attack the general case. Nevertheless, based on computational evidence, the authors have the conviction that the case $b=2^{2^s}$ provides us with the only example in which the considered sequence is eventually periodic (and, in fact, periodic); i.e., we present the conjecture below.

\begin{conj}
The sequence $\mathfrak{S}_b(n)$ is eventually periodic if and only if $b=2^{2^s}$ for some $s\in \mathbb{N}$ and, moreover, in that case the sequence is periodic.
\end{conj}

\end{document}